\documentclass[a4paper]{article}
\usepackage{amsthm,amssymb,amsmath,enumerate}
\usepackage{hyperref}
\hypersetup{
	colorlinks=true,  % instead of frames around the links, colour links
	linkcolor=black,    % colour internal links
	citecolor=black,    % colour citation links
	urlcolor=black      % colour of external links
}

\usepackage{tikz}
\usetikzlibrary{calc}
\usetikzlibrary{backgrounds}
\tikzstyle{hvertex}=[thick,circle,inner sep=0.cm, minimum size=2mm, fill=white, draw=black]
\tikzstyle{hedge}=[very thick]
\tikzstyle{harrow}=[thick,arrows=->]
\tikzstyle{darrow}=[thick,arrows=<-]
\tikzstyle{rededge}=[very thick,red]
\tikzstyle{point}=[draw,circle,inner sep=0.cm, minimum size=1mm, fill=black]
\tikzstyle{pointer}=[thick,->,shorten >=2pt,color=hellgrau]
\tikzstyle{facebdry}=[color=auchblau, very thick] % face boundary
\tikzstyle{face}=[facebdry,fill=hellblau]
\tikzstyle{nface}=[color=hellblau,fill=hellblau,thick] % naked face, without boundary
\tikzset{>={latex}}
\tikzstyle{tinyvx}=[thick,circle,inner sep=0.cm, minimum size=1.3mm, fill=white, draw=black]

\colorlet{auchblau}{blue!60!white}
\colorlet{hellblau}{blue!20!white}
\colorlet{hellrot}{red!40!white}
\colorlet{hellgrau}{black!30!white}

\newtheorem{definition}{Definition}

\newtheorem{theorem}[definition]{Theorem}

\newtheorem{lemma}[definition]{Lemma}

\renewcommand{\phi}{\varphi}
\renewcommand{\epsilon}{\varepsilon}

\newcommand{\comment}[1]{}
\newcommand{\N}{\mathbb N}

\newcommand{\cF}{\mathcal{F}}

\newcommand{\EP}{Erd\H os-P\'osa property}

\title{Directed cycles have the edge-Erd\H os-P\'osa property}
\author{Matthias Heinlein\thanks{Ulm University, {\tt matthias.heinlein@uni-ulm.de}} \and Arthur Ulmer\thanks{Ulm University, {\tt arthur.ulmer@uni-ulm.de}, supported by DFG, grant no.\  BR 5449/1-1}} 
\date{}

\begin{document}

\maketitle

\begin{abstract}
In this short note we prove that for every $k\in \N$ there is a $t_k\in\N$ such that
for every digraph $G$
there are either $k$ edge-disjoint directed cycles in $G$
or a set $X$ of at most $t_k$ edges such that $G-X$ contains no directed cycle.
\end{abstract}

\section{Introduction}
	If $\cF$ is a family of (directed or undirected) graphs, we say that $\cF$ has the vertex/edge-\EP\
	if there is a sequence $(t_k)_{k \in \N}$ such that for every $k\in \N$ and every (di-)graph $G$,
	either $G$ contains $k$ vertex/edge-disjoint subgraphs each isomorphic to a member of $\cF$,
	or there is a set $X$ of at most $t_k$ vertices/edges such that $G-X$ contains no subgraph isomorphic to a member of $\cF$.
	
	Many classes of undirected graphs are known to have the vertex-\EP\ but only few of them are also investigated regarding the edge property. 
	A summary of some results can be found in \cite{RT16} and in a table in \cite{BHJ18}.
	
	When we address the \EP\ in directed graphs only a few results are known --- and nearly all of them regarding the vertex version.
	In 1996, Reed et al. in \cite{Reed1996} proved that directed cycles in digraphs have the vertex-\EP. The bound of $t_k$ in terms of $k$ is extremely large.
	
\begin{theorem}[Reed et al. \cite{Reed1996}]\label{vertexversion}
		For every $k\in \N$ there is a $t_k\in\N$ such that
		for every digraph $G$
		there are either $k$ vertex-disjoint directed cycles in $G$
		or a set $X\subseteq V(G)$ of size at most $t_k$ such that $G-X$ contains no directed cycle.
	\end{theorem}

	To the best of our knowledge the edge-version for directed cycles
	has been an open problem so far.
	In this paper we will prove that it is indeed true.

\section{Proof of the theorem}
	Even et al. \cite{Even1998} showed a correspondence between
	edge-hitting sets in $G$ and vertex-hitting sets in $L(G)$ and for that they also
	 defined the \emph{directed line graph}. We will show an analogous statement for edge-disjoint and 
	 vertex-disjoint directed cycles and can then prove our main theorem.

	Let $G=(V(G),E(G))$ be a digraph. 
	The \emph{directed line graph} $L(G)$ of $G$ has $E(G)$ as its vertex set 
	and two directed edges $(a,b), (c,d)$ of $G$ are joined by a directed edge from $(a,b)$ to $(c,d)$ in $L(G)$
	if and only if $b=c$.
	If $T$ is a subgraph of $L(G)$, we denote by $G[T]$
	the minimal subgraph of $G$ that has edge set $V(T)$.
	For a vertex $u \in V(G)$ we define $E(u)$ as the set of all edges that have $u$ as their first or second vertex. It is easy to see that the subdigraph $L_u=L(G)[ E(u)]$
	%
	% \{(u,v)\in E(D)\}\cup \{(v,u) \in E(D)\}]$ 
	contains no directed cycle.

	\begin{lemma} \label{preimage}
		For every directed cycle $C$ in $G$, $C'=L(G)[E(C)]$ is a directed cycle in $L(G)$.
		For every directed cycle $C'$ in $L(G)$, $C=G[C']$ contains a directed cycle in $G$.
	\end{lemma}
	\begin{proof}
		Let $u_1,\ldots, u_\ell$ be the vertices of $C$. 
		Therefore, the vertices of $C'$ are $v_1=(u_1,u_2),\ldots,v_{\ell}=(u_\ell,u_1)$
		 which are all distinct as $u_1,\ldots, u_\ell$ are distinct.
		The vertices $v_i$ and $v_{i+1 \pmod{\ell}}$ in $L(G)$ are joined by a directed edge in $L(G)$
		because the endvertex of $v_i$ is the startvertex of $v_{i+1 \pmod{\ell}}$.
		Hence, $C'$ is a directed cycle.
		
		Let $C'$ be a directed cycle in $L(G)$ with vertex set $\{v_i=(u_i^1,u_i^2): i=1,\ldots, \ell\}$.
		As $v_i$ and $v_{i+1}$ are joined by a directed edge, we have $u_i^2=u_{i+1 \pmod{\ell}}^1$ for every $i=1,\ldots, \ell$.		
		Now, $G[C']$ has vertex set $\{u_1^1,\ldots, u_\ell^1\}$ where not all vertices are necessarily distinct
		and the edge set consists of all edges $v_i=(u_i^1,u_{i+1 \pmod{\ell}}^1)$, $i=1,\ldots, \ell$.
		Let $i<j$ be integers such that $u_i^1=u_j^2$ and such that $j-i$
		% the size of $\{ v_i,v_{i+1 \pmod{\ell}}, \ldots ,v_{j-1 \pmod{\ell}},v_j\}$ 
		is minimal. It then 
		follows that the graph $G[\{ v_i,\ldots,v_j\}]$ 
		is a directed cycle contained in $G[C']$. 
	\end{proof}
	
	Now we can prove our main theorem.
	
	\begin{theorem}
		For every $k\in \N$ there is a $t_k\in\N$ such that
		for every digraph $G$
		there are either $k$ edge-disjoint directed cycles in $G$
		or a set $X\subseteq E(G)$ of size at most $t_k$ such that $G-X$ contains no directed cycle.
	\end{theorem}
	\begin{proof}
		Apply Theorem \ref{vertexversion} to the directed line graph $L(G)$ of $G$.
		If the theorem returns $k$ vertex-disjoint directed cycles $C_1,\ldots, C_k$ in $L(G)$,
		their preimages $G[C_i]$ in $G$ are edge-disjoint and each contain a directed cycle in $G$, by Lemma \ref{preimage}.
		Hence, $G$ contains $k$ edge-disjoint directed cycles.
		
		If the theorem returns a vertex hitting set $X$ of size $t_k$ in $L(G)$,
		the same set $X$ (as a set of edges) is an edge hitting set for directed cycles in $G$.
		Namely, if $C$ was a directed cycle in $G-X$,
		then, by Lemma \ref{preimage}, the digraph $L(G)[C]$ would be a directed cycle in $L(G)$
		that avoids $X$ contradicting the choice of $X$ as a hitting set in $L(G)$.
	\end{proof}

\section{Discussion}
	Very lately, Kawarabayashi and Kreutzer mentioned in their paper \cite{KK14} on the directed grid theorem that for any $\ell\in \N$, the class of  directed cycles of length at least $\ell$ have the vertex-\EP. So is it possible to use the method above also for long directed cycle?
	Unfortunately not.
	Although the image of a long cycle in $G$ is a long cycle in $L(G)$ again,
	the preimage of a long cycle in $L(G)$ may consist of many short cycles in $G$.
	
	We thank Maximilian F\"urst for asking a question that led to this note.

	\bibliographystyle{amsplain}
	\bibliography{directed}

\end{document}